\documentclass[a4paper]{amsart}

\usepackage{verbatim,amsmath,amssymb,enumitem,amsthm}
\usepackage[displaymath,pagewise]{lineno} %for line numbers switch*,
\usepackage{mathtools,xcolor}

\newtheorem{thm}{Theorem} 

\newtheorem{prop}[thm]{Proposition}
\newtheorem{lem}[thm]{Lemma}
\newtheorem{cor}[thm]{Corollary}
\newtheorem{ex}[thm]{Example}
\newtheorem{rem}[thm]{Remark}

\theoremstyle{remark}
\newtheorem{remark}[thm]{Remark}        % Numbered along with thm
\theoremstyle{plain} 
\newcommand{\thistheoremname}{}
\newtheorem*{genericthm*}{\thistheoremname}
\newenvironment{namedthm*}[1]
{\renewcommand{\thistheoremname}{#1}%
	\begin{genericthm*}}
	{\end{genericthm*}}

\numberwithin{equation}{section} \numberwithin{thm}{section}

\newcommand*\patchAmsMathEnvironmentForLineno[1]{%
	\expandafter\let\csname old#1\expandafter\endcsname\csname #1\endcsname
	\expandafter\let\csname oldend#1\expandafter\endcsname\csname end#1\endcsname
	\renewenvironment{#1}%
	{\linenomath\csname old#1\endcsname}%
	{\csname oldend#1\endcsname\endlinenomath}}%
\newcommand*\patchBothAmsMathEnvironmentsForLineno[1]{%
	\patchAmsMathEnvironmentForLineno{#1}%
	\patchAmsMathEnvironmentForLineno{#1*}}%
\AtBeginDocument{%
	\patchBothAmsMathEnvironmentsForLineno{equation}%
	\patchBothAmsMathEnvironmentsForLineno{align}%
	\patchBothAmsMathEnvironmentsForLineno{flalign}%
	\patchBothAmsMathEnvironmentsForLineno{alignat}%
	\patchBothAmsMathEnvironmentsForLineno{gather}%
	\patchBothAmsMathEnvironmentsForLineno{multline}%
}

\newcommand{\vf}{\varphi}
\newcommand{\cal}{\mathcal}

\newcommand{\dist}{{\rm dist}\,}

\newcommand{\R}{{\mathbb R}}

\newcommand{\ep}{\varepsilon}

\newcommand{\reach}{{\rm reach}\,}

\newcommand{\grad}{\operatorname{grad}}

\newcommand{\inter}{\mathrm{int}}

\def\en{\mathbb N}
\def\er{\mathbb R}

\def\C{\mathcal C}

\newcommand{\graph}{\operatorname{graph}}

\newcommand{\spn}{\operatorname{span}}

\def\halfsq{\hbox{\kern1pt\vrule height 7pt\vrule width6pt height 0.4pt depth0pt\kern1pt}}
\def\ihalfsq{\hbox{\kern1pt \vrule width6pt height 0.4pt depth0pt
		\vrule height 7pt \kern1pt}}

\begin{document}

\title{On sets in $\R^d$ with DC distance function}
\author{Du\v san Pokorn\'y}\author{Lud\v ek Zaj\'i\v cek}

\thanks{The research was supported by GA\v CR~18-11058S}

\begin{abstract}  
	We study closed sets  $F \subset \R^d$ whose distance function $d_F\coloneqq \dist(\cdot,F)$ is DC
	(i.e., is the difference of two convex functions on $\R^d$). Our main result asserts that if $F \subset \R^2$
	is a graph of a DC function $g:\R\to \R$, then $F$ has the above property. If $d>1$, the same holds
	 if $g:\R^{d-1}\to \R$ is semiconcave, however the case of a general DC function $g$ remains open.
\end{abstract}

\email{dpokorny@karlin.mff.cuni.cz}
\email{zajicek@karlin.mff.cuni.cz}

\date{\today}
\maketitle

\section{Introduction}

Let  $F \neq \varnothing$ be a closed subset of $\R^d$ and let $d_F\coloneqq  \dist(\cdot,F)$ be its distance
 function. Recall that a function on $\R^d$ is called DC, if it is the difference of two convex functions. It is  well-known   (see, e.g., \cite[p. 976]{BB}) that
 \begin{equation}\label{odist}
 \text{ the function $(d_F)^2$ is DC but $d_F$ need not be DC.}
\end{equation}
 However, the distance function of some interesting special $F\subset \R^d$
 is DC; it is true for example for  $F$ from Federer's class of sets with positive reach, see \eqref{prdd}.

Our article was motivated by \cite{BB} and by the following question which naturally arises
 in the theory of WDC sets (see  \cite[Question 2, p. 829]{FPR} and \cite[10.4.3]{Fu2}).
\medskip

{\bf Question.} \ 
 Is  $d_F$ a DC function if $F$ is a graph of a $DC$ function $g: \R^{d-1} \to \R$?

\medskip
Note that WDC sets  form a substantial generalization of sets
 with positive reach and still admit the definition of curvature measures (see \cite{PR} or \cite{Fu2}) and $F$ as in Question is a natural example of a WDC set in $\R^d$.

Our main result (Theorem \ref{hlav}) gives the affirmative answer to Question in the case $d=2$;
 the case $d>2$ remains open. However, known results relatively easily imply that the answer
 is positive if $g$ in Question is semiconcave (Corollary \ref{grsemi}).

In \cite{PZ} we show that our main result has some interesting consequences for WDC subsets of $\R^2$, in particular that these sets have DC distance functions.

In Section 2 we recall some notation and needed facts about DC functions. In Section 3 we prove
 our main result (Theorem \ref{hlav}). 
In last Section 4, we prove a number of further results on the system of sets in $\R^d$ which have DC distance function, including Corollary \ref{grsemi} mentioned above. 

We were not able to prove a satisfactory complete characterisation of sets $F\subset \er^2$ with DC distance function, but we believe that our methods and results should lead to such a characterisation. However, in our opinion, the case of $F\subset \er^d$, $d\geq 3$, needs some new ideas.

\section{Preliminaries}
%\subsection{Basic definitions}\label{basic}
 In any vector space $V$, we use the symbol $0$ for the zero element. 
We denote by $B(x,r)$ ($U(x,r)$) the closed (open) ball with centre $x$ and radius $r$.
 The boundary and the interior of a set $M$ are denoted by $\partial M$ and $\inter M$, respectively. A mapping is called $K$-Lipschitz if it is Lipschitz with a (not necessarily minimal) constant $K\geq 0$.

In the Euclidean space $\R^d$,  the norm is denoted by $|\cdot|$ and  the scalar product  by $\langle \cdot,\cdot\rangle$. By $S^{d-1}$ we denote the unit sphere in $\R^d$.

If $x,y\in\R^d$, the symbol $[x,y]$ denotes the closed segment (possibly degenerate). If also 
 $x \neq y$, then  $l(x,y)$  denotes the line
	 joining $x$ and $y$.

The distance function from a set $A\subset \R^d$ is $d_A\coloneqq  \dist(\cdot,A)$ and 
 the metric projection of $z\in \R^d$ to $A$ is   
 $\Pi_A(z)\coloneqq \{ a\in A:\, \dist(z,A)=|z-a|\}$.

  If $f$ is defined in $\R^d$, we use the notation
$f'_+(x,v)$ for the one sided directional derivative of $f$ at $x$ in direction $v$.

\medskip
%\subsection{DC functions}\label{dc}

 Let $f$ be a real function defined on an open convex set $C \subset \R^d$. Then we say
	 that $f$ is a {\it DC function}, if it is the difference of two convex functions. Special DC
	 functions are semiconvex and semiconcave functions. Namely, $f$ is a {\it semiconvex} (resp.
	 {\it semiconcave}) function, if  there exist $a>0$ and a convex function $g$ on $C$ such that
	$$   f(x)= g(x)- a \|x\|^2\ \ \ (\text{resp.}\ \  f(x)= a \|x\|^2 - g(x)),\ \ \ x \in C.$$

We will use the following well-known properties of DC functions.

\begin{lem}\label{vldc}
Let $C$ be an open convex subset of $\R^d$. Then the following assertions hold.
\begin{enumerate}
\item[(i)]
	If $f: C\to \R$ and $g: C\to \R$ are DC, then (for each $a\in \R$, $b\in \R$) the functions $|f|$, $af + bg$, 
	 $\max(f,g)$ and $\min(f,g)$ are DC.
	\item[(ii)]
	Each locally DC  function  $f:C \to \R$  is DC. 
		\item[(iii)]	Each  DC  function  $f:C \to \R$  is 
  Lipschitz on each compact convex set $Z\subset C$.
	\item[(iv)]
Let $f_i: C \to \R$, $i=1,\dots,m$, be DC functions. Let $f: C \to \R$ be a continuous function
 such that $f(x) \in \{f_1(x),\dots,f_m(x)\}$ for each $x \in C$. Then $f$ is DC on $C$.
 	\item[(v)] Each $\C^2$ function $f:C\to\er$ is DC.
\end{enumerate}
\end{lem}
\begin{proof}
Property (i) follows easily from definitions, see e.g. \cite[p. 84]{Tuy}. Property (ii) was proved in \cite{H}.
 Property (iii) easily follows from the local Lipschitzness of  convex functions. Assertion (iv) is a special case of \cite[Lemma 4.8.]{VeZa} (``Mixing lemma'').
 To prove (v) observe that (e.g. by \cite[Prposition~1.1.3~(d)]{CS}) each $C^2$ function is locally semiconcave and therefore locally DC, hence, DC by (ii).
\end{proof}

By well-known properties of convex and concave functions, we easily obtain that each locally DC function $f$ on an open set $U \subset \R^d$ has all one-sided directional derivatives finite and
\begin{equation}\label{zlose}  
 g_+'(x,v) + g_+'(x,-v) \leq 0,\ \ x \in U, v\in \R^d,\ \ \ \text{if}\ \ g\ \ \text{is locally semiconcave on}\ \ U.
\end{equation}

%First note that the distance function $d_A$ and the  metric projection $\Pi_A$ were defined in
%Subsection~\ref{basic}.

Recall that if $\varnothing \neq A\subset \R^d$ is closed, then $d_A$ need not be DC; however (see, e.g., \cite[Proposition 2.2.2]{CS}), 
\begin{equation}\label{loksem}
\text{$d_A $ is locally semiconcave (and so locally DC) on $\R^d \setminus A$.}
\end{equation}

\section{Main result}\label{graf}

In the proof of Theorem \ref{hlav} below we will use the following simple ``concave mixing lemma''.
\begin{lem}\label{comix}
Let $U \subset \R^d$ be an open convex set and let $\gamma: U \to \R$  have finite one-sided directional
 derivatives $\gamma_+'(x,v)$, ($x \in U, \ v\in \R^d$).  Suppose that 
\begin{equation}\label{nezlo}
   \gamma_+'(x,v) + \gamma_+'(x,-v) \leq 0,\ \ x \in U, v\in \R^d,
\end{equation}
 and that
\begin{multline}\label{pokrc}
\text{ $\graph \gamma$ is covered by graphs of a finite number}\\ 
\text{of concave functions defined  on $U$.}
  \end{multline}
  Then $\gamma$ is a concave function.
\end{lem}
\begin{proof}
Since $\gamma$ is clearly concave if each function $t\mapsto \gamma(a + tv),\ (a\in C, v \in S^{d-1})$  
 is concave on its domain, it is sufficient to prove the case $d=1$, $C=(a,b)$.
Set $h(x)\coloneqq  -\gamma(x),\ x \in (a,b)$; we need to prove that $h$ is convex. Observe that
 \eqref{nezlo} easily implies the condition
\begin{equation}\label{prle}
h'_-(x)\leq h'_+(x),\ \ \ x \in (a,b).
\end{equation}
 and \eqref{pokrc} implies that there exists a finite set $\{h_{\alpha}:\ \alpha\in A\}$ of convex functions on 
  $(a,b)$ such that $\graph h \subset \bigcup\{\graph h_{\alpha}:\ \alpha \in A\}$.
To prove the convexity of $h$, it is sufficient to show that the function $h'_+$ is nondecreasing on
 $(a,b)$ (see e.g.  \cite[Chap. 5, Prop. 18, p. 114]{Roy}); equivalently (it follows e.g. from \cite[Chap. IX, \S7, Lemma 1, p. 266]{Na}) to prove  that
\begin{multline}\label{movbo}
\forall x_0 \in (a,b)\ \exists \delta>0\ \forall x: \ (x \in (x_0,x_0+\delta) \Rightarrow h'_+(x)\geq h'_+(x_0))
\\ \wedge\ \  (x \in (x_0-\delta,x_0) \Rightarrow h'_+(x)\leq h'_+(x_0)).
\end{multline} 
So suppose, to the contrary, that \eqref{movbo} does not hold; then there exists a sequence
 $x_n \to x_0$ such that either
\begin{equation}\label{nale}
x_n < x_0\ \ \text{and}\ \ h'_+(x_n) > h'_+(x_0)\ \ \text{for each}\ \ n \in \en
\end{equation}  
 or
\begin{equation}\label{napr}
x_n > x_0\ \ \text{and}\ \ h'_+(x_n) < h'_+(x_0)\ \ \text{for each}\ \ n \in \en.
\end{equation} 
Since $h$ is clearly continuous, each set $F_{\alpha}\coloneqq  \{x \in (a,b):\ h_{\alpha}(x)=h(x)\}$,
 $\alpha \in A$, is closed in $(a,b)$. Since $A$ is finite, it is easy to see that for each $n \in \en$ there exists $\alpha(n) \in A$ such that $x_n \in F_{\alpha(n)}$ and $x_n$ is a right accumulation
 point of  $ F_{\alpha(n)}$.  Using finiteness of $A$ again, we can suppose that there exists $\alpha \in A$
 such that $\alpha(n)= \alpha$, $n \in \en$ (otherwise we could consider a subsequence of $(x_n)$).
 
Now suppose that \eqref{nale} holds. Since $x_n \in F_{\alpha},\ n=0,1,\dots,$ we obtain
 that $h'_+(x_n)= (h_{\alpha})_+'(x_n),\ n \in \en$, and $h'_-(x_0) = (h_{\alpha})_-'(x_0)$. 
Using also the convexity of $h_{\alpha}$ and \eqref{prle}, we obtain
$$ h'_+(x_n)= (h_{\alpha})_+'(x_n) \leq  (h_{\alpha})_-'(x_0) =  h'_-(x_0) \leq h'_+(x_0),$$
 which contradicts \eqref{nale}. Since the case when \eqref{napr} holds is quite analogous,
 neither \eqref{nale} nor \eqref{napr} is possible and so we are done.
\end{proof}
We will need also the following easy lemma.
 \begin{lem}\label{kofu}
	 Let $V$ be a closed angle in $\R^2$ with vertex $v$ and measure $0<\alpha< \pi$.
	 Then  there exist an affine function $A$ on $\R^2$ and  a concave function $\psi$ on $\R^2$ which is Lipschitz with constant
	 $\sqrt{2}\tan (\alpha/2)$ such that  $|z-v| + \psi(z)= A(z),\ z \in V$.
	\end{lem}
	\begin{proof}
	We can suppose without any loss of generality that $v= (0,0)$ and
	$$  V= \{(x,y):\ x \geq 0, |y| \leq x\, \tan(\alpha/2)\}.$$
	Then $|z-v| = \sqrt{x^2+y^2}$ for $z=(x,y)$. Define the convex function
	$$\vf(x,y)\coloneqq  \sqrt{x^2+y^2} - x,\ (x,y)\in V.$$
	 We will show that
	\begin{equation}\label{jelip}
	\vf\ \ \text{is Lipschitz with constant}\ \ \sqrt{2}\tan (\alpha/2).
	\end{equation}
	To this end estimate, for $(x,y) \in  \inter\, V$,
	$$ \left|\frac{\partial \vf}{\partial x}(x,y)\right| = \left| \frac{x}{\sqrt{x^2+y^2}}-1\right|
	 =  \frac{y^2}{ ( x + \sqrt{x^2+y^2}) \sqrt{x^2+y^2} } \leq \frac{|y|}{x} \leq \tan (\alpha/2),
$$	
	$$ \left|\frac{\partial \vf}{\partial y}(x,y)\right| = 
	  \frac{|y|}{  \sqrt{x^2+y^2} } \leq \frac{|y|}{x} \leq \tan (\alpha/2).
$$	
 Thus $|\grad \vf (x,y)| \leq \sqrt{2}\tan (\alpha/2)$ for $(x,y) \in  \inter\, V$ 
 and \eqref{jelip} follows. So $\vf$ has a convex extension $\tilde \vf$ to $\R^2$
 which is also Lipschitz with constant  $\sqrt{2}\tan (\alpha/2)$ (see, e.g., \cite[Theorem 1]{CM}).
  Now we can  put $\psi\coloneqq  - \tilde \vf$, since $\sqrt{x^2+y^2} + \psi(x,y) =x=: A(x,y),\ (x,y) \in V$.
\end{proof}

	\begin{thm}\label{hlav}
	 Let $f: \R \to \R$ be a DC function. Then the distance function $d\coloneqq  \dist(\cdot, \graph f)$
	 is DC on $\R^2$.
	\end{thm}
	\begin{proof}
	By \eqref{loksem}, $d$ is  locally DC on
	 $\R^2 \setminus \graph f$. So, by Lemma \ref{vldc} (ii), 
	 it is sufficient to prove that, for each $z \in \graph f$, the distance function $d$ is DC
	 on a convex neighbourhood of $z$. Since we can clearly suppose that $z= (0, f(0))$,
 it is sufficient to prove that 
	\begin{equation}\label{suff}
	\text{$d$ is DC on $U\coloneqq U((0,f(0)), 1/10)$.}
	\end{equation}
	
	 Write $f= g - h$, where $g$, $h$ are convex functions on $\R$. For each $n \in \en$, consider
	 the equidistant partition $D_n= \{x^n_0= -1 < x^n_1< \dots
	 < x^n_n=1\}$ of $[-1,1]$. Let $g_n$, $h_n$ be the piece-wise linear function on $[-1,1]$
	 such that $g_n(x^n_i) =  g(x^n_i)$, $h_n(x^n_i) =  h(x^n_i)$ ($0\leq i \leq n$) and 
	 $g_n$, $h_n$ are affine on each interval $[x_{i-1}, x_i]$ ($1\leq i \leq n$).
	 Put $f_n\coloneqq  g_n - h_n$ and $d_n\coloneqq  \dist(\cdot, \graph f_n)$. Choose $L > 0$ such that both
	$g\restriction_{[-1,1]}$ and $h\restriction_{[-1,1]}$ are $(L/2)$-Lipschitz and observe 
	 that all $g_n$, $h_n$, $f_n$
	  are  $L$-Lipschitz. Since  $f_n$ uniformly converge to $f$ on $[-1,1]$, we easily
		 see that $d_n \to d$ on $\overline{U}$.

	Choose an integer $n_0$ such that
	\begin{equation}\label{nnula}
	n_0 \geq 6 \ \ \text{and}\ \ |f_n(0)- f(0)| < \frac{1}{10}\ \ \text{for each}\ \ n \geq n_0.
	\end{equation}
	
	 We will prove that there exist $L^* >0$ and concave functions $c_n$ ($n\geq n_0$) on 
	 $\overline{U}$ 
	 such that
	\begin{equation}\label{lcn}
	 \text{each}\ \ c_n\ \ \text{is Lipschitz with constant}\ \ L^*\ \ \text{and}
	\end{equation}
	\begin{equation}\label{dnpcn}
	c_n^*\coloneqq  d_n + c_n\ \ \text{is concave on}\ \ \overline{U}.
	\end{equation}
  Then we will done, since 	\eqref{lcn} and \eqref{dnpcn} easily imply \eqref{suff}. Indeed,
	 we can suppose that $c_n((0,f(0)))=0$ and, using Arzel\`a-Ascoli theorem, we obtain that
	 there exists an increasing sequence of indices $(n_k)$ such that $c_{n_k} \to c$, where
	 $c$ is a continuous concave function on  $\overline{U}$. So
	$d_{n_k} + c_{n_k} \to d+c=:c^*$  on  $\overline{U}$. Using \eqref{dnpcn}, we obtain that  $c^*$ is concave and thus $d=c^*-c$ is DC
	 on $U$.

	To prove the existence of $L^*$ and $(c_n)$, fix an arbitrary $n\geq n_0$. For brevity
	 denote   $\Pi\coloneqq \Pi_{\graph f_n}$ and 
	 put    $x_i\coloneqq x^n_i$,\ 
	 $z_i\coloneqq  (x_i, f_n(x_i)),\ i=0,\dots,n$.  For  $i=1,\dots,n-1$, let  $0 \leq \alpha_i < \pi$
	 be the angle between the vectors $z_i - z_{i-1}$ and $z_{i+1}- z_i$.
	 Denote  
	$$s_i: = \frac{f_n(x_{i+1})- f_n(x_i)}{x_{i+1}-x_i}\ \text{and}\ \beta_i\coloneqq  \arctan s_i,\ \ 
	 i=0,\dots,n-1.$$
	Then clearly $\alpha_i= |\beta_i - \beta_{i-1}|$. One of the main ingredients of the present proof  is the easy fact
	 that
	 \begin{equation}\label{ssi}
	 \sum_{i=1}^{n-1} |s_i- s_{i-1}| \leq 4L.
	 \end{equation}
	 It immediately follows from the well-known estimate of (the ``convexity'') $K_a^b(f_n)$ (see \cite[p. 24, line 5]{RoVa}).
	  To give, for completeness, a direct proof, denote
	  $$ \tilde s_i: = \frac{g_n(x_{i+1})- h_n(x_i)}{x_{i+1}-x_i},\ \ s^*_i: = \frac{h_n(x_{i+1})- h_n(x_i)}{x_{i+1}-x_i},
	   \ \ i=0,\dots,n-1,$$
	    and observe that the finite sequences $(\tilde s_i)$, $(s^*_i)$ are nondecreasing. Consequently
	    $$  \sum_{i=1}^{n-1} |\tilde s_i- \tilde s_{i-1}| = \tilde s_n - \tilde s_1 \leq 2L \ \ \text{and}\ \
	     \sum_{i=1}^{n-1} | s^*_i-  s^*_{i-1}| =  s^*_n -  s^*_1 \leq 2L. 
	    $$ 
	     Since $s_i= \tilde s_i - \tilde s^*_i$, \eqref{ssi} easily follows.

	 Since 
	$$ \alpha_i = |\beta_i - \beta_{i-1}| \leq |\tan(\beta_i) - \tan(\beta_{i-1})|= |s_i- s_{i-1}|,$$
	 we obtain
	\begin{equation}\label{odsa}
	\sum_{i=1}^{n-1} \alpha_i \leq 4L.
	\end{equation}
	Since  $|\beta_i| \leq \arctan L$, we have $\alpha_i/2 \leq \arctan L$. Further, since 
	 the function $\tan$ is convex on $[0, \pi/2)$, the function $s(x)= \tan x/x$ is increasing on
	  $(0, \pi/2)$. These facts easily imply
	$$  \tan \left(\frac{\alpha_i}{2}\right) \leq \frac{\alpha_i}{2}\cdot \frac{L}{\arctan L}.$$
 Thus we obtain by \eqref{odsa}
\begin{equation}\label{otap}
\sum_{i=1}^{n-1} \sqrt{2} \tan \left(\frac{\alpha_i}{2}\right) \leq   \frac{2\sqrt{2}L^2}{\arctan L} =: M.
\end{equation}

Further observe  that each $d_n$ is DC on $\R^2$ and consequently
		\begin{equation}\label{exdir}
		(d_n)_{+}'(x,v) \in \R\ \ \ \text{exists for every}\ \ \ x,v \in \R^2.
		\end{equation}
		Indeed, since each segment $[z_{i-1},z_i]$ is a convex set, by the well known fact the distance functions  $\dist(\cdot, [z_{i-1}, z_i])$, $i=1,\dots,n,$ are convex and consequently $d_n$ is DC by \eqref{sjedno} below.
		%Lemma \ref{vldc} (iv) (``mixing lemma'').

If $\alpha_i \neq 0$, set
 $$V_i\coloneqq  \{ z \in \R^2:\ \langle z-z_i, z_{i+1}-z_i\rangle \leq 0,\ \langle z-z_i, z_{i-1}-z_i\rangle \leq 0 \},$$
  which is clearly a closed angle  with vertex  $z_i$	
 and measure $\alpha_i$. Let $\psi_i$ and $A_i$ be the (concave and affine) functions on
 $\R^2$ which correspond to $V_i$ by Lemma \ref{kofu}.
If $\alpha_i =0$,  put $\psi_i(z)\coloneqq 0,\ z \in \R^2$.

 Now set
$$ \eta_n: = \sum_{i=1}^{n-1}  \psi_i.$$
Then  $\eta_n$ is a concave function on $\R^2$ and Lemma \ref{kofu} with \eqref{otap} imply that
\begin{equation}\label{lipeta}
  \eta_n\ \ \text{ is Lipschitz with constant}\ \ M,
  \end{equation}
   and, if $\alpha_i \neq 0$,
 \begin{equation}\label{kompvi}
|z - z_i| + \psi_i(z) = A_i(z),\  z\in  V_i.
\end{equation}

 The concave function $c_n$ with properties  \eqref{lcn},  \eqref{dnpcn} will be defined
 as  $c_n(x)\coloneqq  \eta_n(x) + \xi_n(x),\ x \in \overline{U}$, where the concave function $\xi_n$  on $A\coloneqq  (-1,1) \times \R$ will be defined to
 ``compensate the non-concave behaviour of $d_n$ at points of $\graph f_n$'' in the sense
 that, for each point $z \in A \cap \graph f_n$,
\begin{equation}\label{kompgr}
 (d_n + \xi_n)_+'(z,v) +  (d_n + \xi_n)_+'(z,-v)  \leq 0\ \ \ \text{whenever}\ \ \ v \in \R^2.
\end{equation}
We set, for $(x,y) \in A$,
$$ \xi_n(x,y)\coloneqq  - \max(2g_n(x)-y, 2h_n(x)+y)\ \ \ \text{and}\ \ \ p_n(x,y)\coloneqq  |f_n(x)-y|.$$
Obviously, 
\begin{equation}\label{lipxi}
 \text{$\xi_n$ is concave and Lipschitz with constant $2L+1$.}
 \end{equation}
Further, for $(x,y) \in A$,
\begin{multline*}
 p_n(x,y) = \max(g_n(x)-h_n(x)-y, h_n(x)-g_n(x)+y)\\ = \max(2g_n(x)-y, 2 h_n(x)+y)
  - h_n(x) - g_n(x), 
	\end{multline*}
which shows that $p_n$ is a DC function and $p_n + \xi_n$ is concave. Consequently, for
  each $z\in A$ and $v \in \R^2$,
		\begin{equation}\label{pxi}
		(p_n)'_+(z,v) + (\xi_n)'_+(z,v) + (p_n)'_+(z,-v) + (\xi_n)'_+(z,-v)  \leq 0.
		\end{equation} 

  Since, for each point
 $z \in \graph f_n \cap A$, we have $d_n(z)= p_n(z)=0$ and for each $(x,y) \in A$ clearly
 $d_n(x,y) \leq |(x,y) - (x,f_n(x))| = p_n(x,y)$, we easily obtain (for each $v \in \R^2$)
\begin{equation}\label{troj}
 (p_n)'_+(z,v) +  (p_n)'_+(z,-v)  \geq  (d_n)'_+(z,v) +  (d_n)'_+(z,-v),
\end{equation}
  which, together with \eqref{pxi},
  implies   \eqref{kompgr}.

Now set
$$c_n(x)\coloneqq  \eta_n(x) + \xi_n(x),\ x \in \overline{U}.$$
By \eqref{lipeta} and \eqref{lipxi} we obtain that \eqref{lcn} holds with $L^*\coloneqq  M + 2L+1$.

To prove \eqref{dnpcn}, it is clearly sufficient to show that $\gamma=c_n^*: = d_n + c_n$
 is concave on $U$; we will prove it by Lemma \ref{comix}. 
  
 First we verify the validity
 of \eqref{nezlo} for each $z \in U$.  If $z \notin \graph f_n$, then \eqref{nezlo} holds by  \eqref{zlose}, since 
  $\gamma = d_n +  \eta_n + \xi_n$ on $U$, $d_n$ is locally semiconcave on $\R^2 \setminus  \graph f_n$ and $\eta_n+ \xi_n  $
   is concave on $U$.  If  $z \in \graph f_n$, then \eqref{nezlo} follows by \eqref{kompgr} and the concavity of $\eta_n$ on $U$.
   
 So it is sufficient to verify \eqref{pokrc}. To this end, first define on $U$ the functions
 $$   \omega_i\coloneqq  \dist(\cdot, l(z_i,z_{i+1}))\ \ \text{and}\ \ \mu_i\coloneqq  \omega_i+ \eta_n + \xi_n,\ \ \ \ i=1,\dots,n-1. $$
 Since each $\graph \omega_i$ is covered by graphs of two affine functions, we see 
 that 
 \begin{equation}\label{podvco}
 \text{$\graph \mu_i$ is covered by graphs of two concave functions.}
 \end{equation}
 
   Now consider an arbitrary $z\in U$
 and choose a point  $z^* \in  \Pi(z)$. 
 Since $d_n(z) \leq 1/5$ by  \eqref{nnula} and $n \geq n_0 \geq 6$,
 we obtain  $z^* \in \bigcup_1^{n-2} [z_i,z_{i+1}]$. 
 
 If $z^*= z_i$ for some $1\leq i \leq n-1$
 with $\alpha_i \neq 0$, then we easily see that $z\in V_i$ and $d_n(z)= |z-z_i|$, and consequently
$$  \gamma(z)=  (|z-z_i| + \psi_i(z)) + \sum_{1\leq j\leq n-1, j\neq i} \psi_j(z) +\xi_n(z)       = \nu_i(z),$$
where
$$\nu_i(z)\coloneqq  A_i(z) + \sum_{1\leq j\leq n-1, j\neq i} \psi_j(z) +\xi_n(z) ,\ \ z \in U,$$
 is concave on $U$.

If  $z^*= z_i$ and $\alpha_i=0$, or $z^*\in [z_i,z_{i+1}]\setminus\{z_i, z_{i+1}\}$ for some $1\leq i \leq n-1$, then
 clearly $d_n(z) = \omega_i(z)$ and so $\gamma(z) = \mu_i(z)$.

So we have proved that the graph of $\gamma=c^*_n$ is covered by graphs of functions $\nu_i,\  1\leq i \leq n-1$, $\alpha_i \neq 0$,  and  functions $\mu_i,\  1\leq i \leq n-1$.  Using  \eqref{podvco}, we obtain \eqref{pokrc}
 and  Lemma \ref{comix} implies that $\gamma=c^*_n$ is concave.
	\end{proof}

\section{Other  results}\label{other}
We finish the article with a number of additional results on the systems
$$ \cal D_d\coloneqq  \{\varnothing\} \cup  \{\varnothing \neq A \subset \R^d: \ A\ \ \text{is closed}\ \ \text{and}\ \  
d_A\ \ \text{is DC}\},\ \ \ d=1,2,\dots.                $$
First we observe that the description of $\cal D_1$ is very simple since
\begin{equation}\label{cald1}
\text{$A \subset \R$ belongs to $\cal D_1$ iff the system of all components of $A$
 is locally finite.}
\end{equation}
Indeed, if the system of all components of $\varnothing \neq A\subset \R$ is locally finite,
 then Lemma \ref{vldc} (ii) easily implies that $d_A$ is DC.

If the system of all components of $A$
is not locally finite, then there exists a sequence $(c_n)$ of centres of components of
 $\R\setminus A$ converging to a point $a\in A$. Therefore $d_A$ is not one-sidedly strictly
 differentiable at $a$, since $(d_A)'_{\pm}(c_n) = \mp1$. Consequently 
 $d_A$ is not DC, since each DC function on $\R$ is  one-sidedly strictly
 differentiable at each point (see \cite[Note 3.2]{VeZa} or \cite[Proposition 3.4(i) together
 with Remark 3.2]{VZ2}).

From this characterisation easily follows that $\cal D_1$ is closed with respect to finite unions and intersections and that, for a closed set $M\subset \er$,
\begin{equation}\label{eq:boundaryEquivalence}
M\in \cal D_1\iff \partial M\in\cal D_1.
\end{equation}

Concerning $d\geq 2$ further observe that 
\begin{equation}\label{sjedno}
\cal D_d\ \ \text{is closed with respect to finite unions}.
\end{equation}
Indeed, if $\varnothing \neq  A \in \cal D_d$ and $\varnothing \neq  B \in \cal D_d$, then
$d_{A\cup B} = \min (d_A, d_B)$ and so $d_{A\cup B}$ is DC by Lemma \ref{vldc} (i). 

Example~\ref{ex:boundaryAndIntersection} below shows that already $\cal D_2$ is not closed with respect to finite intersections.
Equivalence \eqref{eq:boundaryEquivalence} does not generalize already to dimension $2$ either (see again Example~\ref{ex:boundaryAndIntersection}), however, one can see that, for a closed set $M\subset \er^d$,
\begin{equation}\label{eq:boundaryEquivalenceGeneral}
\partial M \in \cal D_d\iff (M \in \cal D_d \quad\text{and}\quad \overline{\R^d \setminus M}\in  \cal D_d),\quad d\in\en.
\end{equation}
%To prove that first denote $d_{\partial M}=\dist(\cdot,\partial M)$,
%$d_{M}=\dist(\cdot,M)$ and 
%$d_{\overline{\R^d \setminus M}}=\dist(\cdot,\overline{\R^d \setminus M})$.
To prove one implication suppose $\partial M\in \cal D_d$. If $x\not\in M$ then clearly $\Pi_M(x)\in\partial M$ and so $d_{M}(x)=d_{\partial M}(x)$.
Consequently, for each $x\in\er^d$, $d_{M}(x)\in\{0,d_{\partial M}(x)\}$ and so $M\in\cal D_d$ by Lemma~\ref{vldc}~(iv). Similarly, if $x\not\in \overline{\R^d \setminus M}$ then $\Pi_{\overline{\R^d \setminus M}}(x)\in\partial (\overline{\R^d \setminus M})\subset\partial M$ so again $d_{\overline{\R^d \setminus M}}(x)\in\{0,d_{\partial M}(x)\}$ and $\overline{\R^d \setminus M}\in  \cal D_d$ follows.

To prove the opposite implication it is enough to show that $d_{\partial M}=\max(d_M,d_{\overline{\R^d \setminus M}})$ if $\partial M\not=\varnothing$.
Clearly $d_{\partial M}\geq \max(d_M,d_{\overline{\R^d \setminus M}})$, since 
$\partial M= M\cap\overline{\R^d \setminus M}$.
To prove the opposite inequality suppose to the contrary that $$r\coloneqq d_{\partial M}(x)> \max(d_M(x),d_{\overline{\R^d \setminus M}}(x))$$ for some $x\in\er^d$.
Consequently $U(x,r)\cap M\not=\varnothing$ and $U(x,r)\cap\overline{\R^d \setminus M}\not=\varnothing$.
Then also $U(x,r)\cap\R^d \setminus M\not=\varnothing$ and thus $U(x,r)\cap\partial M\not=\varnothing$ which is a contradiction.

%pick $x\in\er^d$ and put $r=d_{\partial M}(x)$.
%If $x\in\partial M$ we can again use $\partial M= M\cap\overline{\R^d \setminus M}$ to obtain that 
%$d_{\partial M}=\max(d_M,d_{\overline{\R^d \setminus M}})=0$, so assume that $r>0$.
%Then $U(x,r)\cap\partial M=\varnothing$ and so either $U(x,r)\cap M=\varnothing$, or $U(x,r)\cap\overline{\R^d \setminus M}=\varnothing$ holds.
%In either case, the inequality 
%$d_{\partial M}=r\leq \max(d_M,d_{\overline{\R^d \setminus M}})$ follows.

Before presenting the following example we first observe that the function $g(x)=x^5\cos\frac{\pi}{x}$, $x\not=0$, $g(0)=0$, is $\C^2$ on $\er$ and therefore DC by Lemma~\ref{vldc}~(v).
Indeed, a direct computation shows that 
$g''(x)=x\left(8\pi x\sin\frac{\pi}{x}-(\pi^2-20x^2)\cos\frac{\pi}{x}\right)$, $x\not =0$, and $g''(0)=0$.

\begin{ex}\label{ex:boundaryAndIntersection}
	There are sets $A,B\in\cal D_2$ such that $A\cap B\not\in\cal D_2$. Further,
	there is a set $K\in {\cal D}_2$ such that $\partial K\not\in{\cal D}_2$ and 
	$\overline{\er^2\setminus K}\not\in{\cal D}_2$.
	
\end{ex}
\begin{proof}
	Define $g(x)=x^5\cos\frac{\pi}{x}$, $x\not=0$, $g(0)=0$, and $f(x)=0$, $x\in \er$. Put 
	$$
	A=\{ (x,y): y\geq f(x) \},\; 
	B=\{ (x,y): y\leq g(x) \},\; H=\{x: f(x)\leq g(x) \}.
	$$
	Since both $f$ and $g$ are DC, we obtain that $A,B\in \cal D_2$ by Theorem~\ref{hlav} and \eqref{eq:boundaryEquivalenceGeneral}.
	Put $M=A\cap B$ and $K=\overline{\er^2\setminus M}$.
	Clearly also $M=\overline{\er^2\setminus K}$.
	First note that $M\not\in\cal D_2$ since the function $x\mapsto d_M(x,0)$ is equal to $d_H$, but clearly $H\not\in \cal D_1$ by\eqref{cald1}.
	
	We obtain that $K\in\cal D_2$ by \eqref{sjedno}, since $K=C\cup D$, where $C=\{ (x,y): y\leq f(x) \}$ and $D=\{ (x,y): y\geq g(x) \}$, and $C,D\in \cal D_2$ by Theorem~\ref{hlav} and \eqref{eq:boundaryEquivalenceGeneral}.
	Finally, $\partial K\not\in\cal D_2$ by \eqref{eq:boundaryEquivalenceGeneral} applied to $K$.
\end{proof}
	
%%	but $M,\partial K\not\in\cal D_2$ since the function $x\mapsto d_M(x,0)$ is equal to $d_H$, but clearly $H\not\in \cal D_1$ (and similarly for $d_{\partial K}$).
%	The second part follows by putting $A=U_f$ and $B=L_g$.	

Now we will show that equivalence \eqref{eq:boundaryEquivalence} holds for sets $M$ of positive reach (cf. \eqref{prdd}).
We first recall their definition.

If $A\subset \R^d$ and $a\in A$, we define
$$\reach(A,a)\coloneqq \sup\{r\geq 0:\, \Pi_A(z)\ \text{is a singleton for each}\ \ z \in U(a,r)\}$$
and the reach of $A$ as
$$\reach A\coloneqq \inf_{a\in A}\reach (A,a).$$
Note that each set with positive reach is clearly closed.

As mentioned in Introduction, it is essentially well-known that
\begin{equation}\label{prdd}
\text{if $A \subset \R^d$ has positive reach, then $A \in \cal D_d.$}
\end{equation}
Indeed, for each $a\in A$ \cite[Proposition 5.2]{CH}
implies that $d_A$ is semiconvex on $U(a, \reach A/2)$, which with \eqref{loksem} and
Lemma \ref{vldc} (ii) implies that $d_A$ is DC.

%
%As far as we know, the following  observation
%is new.

\begin{prop}\label{doplpr}
	Let $\varnothing \neq A \subset \R^d$ be a set with positive reach and  $B\coloneqq  \overline{\R^d \setminus A}$. Then both $B$ and $\partial A$ belong to $\cal D_d$.
\end{prop}
\begin{proof}
	By \eqref{prdd} and \eqref{eq:boundaryEquivalenceGeneral} it is sufficient to prove that $B\in\cal D_d$.
	Since $d_B$ is locally DC on $\R^d\setminus B$ (see \eqref{loksem}) and on $\inter B$ 
	(trivially), by  Lemma \ref{vldc} (ii) it is sufficient
	to prove that
	\begin{equation}\label{lokhr}
	\text{for each $a \in \partial B$ there exists $\rho>0$ such that $d_B$ is DC on $U(a,\rho)$.}
	\end{equation}
	To prove \eqref{lokhr}, choose  $0<r< \reach A$ and denote $A_r\coloneqq  \{x:\ \dist(x,A)=r\}$.
	We will first prove that
	\begin{equation}\label{soucet}
	\dist(x,B) +r =  \dist(x, A_r),\ \ \text{whenever}\ \ x \in \R^d \setminus B = \inter A.
	\end{equation}
	To this end, choose an arbitrary $x \in \inter A$. Obviously, there exists $y\in  \partial B \subset \partial A$ such that $\dist(x,B) = |x-y|$. Since $A$ has positive reach and $y \in \partial A$,
	there exists  $z \in A_r$ such that $|y-z| = r$ (It follows, e.g., from \cite[Proposition 3.1 (v),(vi)]{RZ}). Therefore
	$$ \dist(x, A_r) \leq |x-z| \leq |x-y| + |y-z| = \dist(x,B) +r.$$
	To prove the opposite inequality, choose a point $z^* \in A_r$ such that $\dist(x, A_r) = |x-z^*|$.
	Obviously, on the segment  $[x,z^*]$ there exists a point $y^* \in \partial A \subset B$. Then
	$$  \dist(x, A_r) = |x-z^*| = |x-y^*| + |y^* - z^*| \geq \dist(x,B) +r,$$
	and \eqref{soucet} is proved.
	
	Now let $a \in \partial B \subset \partial A$ be given. Then $a \notin A_r$ and so by \eqref{loksem}
	there exists $\rho>0$ such that $\dist(\cdot, A_r)$ is DC on $U(a,\rho)$. For $x \in U(a,\rho)$,
	$d_B(x) = \dist(x, A_r)-r$ if $x \in \inter A$ (by \eqref{soucet}) and $d_B(x)=0$ if $x \notin \inter A$. Thus  Lemma \ref{vldc} (iv) implies that $d_B$ is DC on $U(a,\rho)$, which proves \eqref{lokhr}.
\end{proof}

Further recall that our main result (Theorem \ref{hlav}) asserts that
\begin{equation}\label{grdd}
\graph g \in \cal D_2\ \ \ \text{whenever}\ \ \ g:\R \to \R\ \ \text{is}\ \ DC.
\end{equation}

 Motivated by a natural question, for which non DC functions $g$ \eqref{grdd} holds,
 we present the following result, whose proof is implicitly contained in the proof
 of \cite[Proposition 6.6]{PRZ}; see Remark \ref{mani} below.
\begin{prop}\label{lidc}
If $g: \R^{d-1} \to \R$ ($d\geq 2$) is locally Lipschitz and $A\coloneqq  \graph g \in \cal D_d$,
 then $g$ is DC.
\end{prop}
\begin{remark}\label{mani}
One implication of \cite[Proposition 6.6]{PRZ} gives that if $A$ is as in Proposition \ref{lidc}
 (or, more generally, $A$ is a  Lipschitz manifold of dimension $0<k<d$; see \cite[Definition 2.4]{PRZ} for this notion) and $A$ is WDC, then $g$ is DC (or is a DC manifold of dimension $0<k<d$, respectively). The proof of this implication works with an aura $f=f_M$ of a set $M$, but 
 under the assumption that $A \in \cal D_d$, the proof clearly also works, if we use the distance function $d_A$ instead of $f$. So we obtain not only Proposition \ref{lidc}, but also
 the following more general result.

{\it If $A \subset \R^d$ is a  Lipschitz manifold of dimension $0<k<d$ and $A \in \cal D_d$,
 then $A$ is a DC manifold of dimension $k$.}
\end{remark}

Recall that it is an open question, whether $\graph g\in \cal D_d$, whenever 
  $g:\er^{d-1}\to\er$ is a DC function. However, using Proposition \ref{doplpr},
	 we easily obtain:
	\begin{cor}\label{grsemi}
	If $g:\er^{d-1}\to\er$ is a semiconcave function then $\graph g \in \cal D_d$.
	\end{cor}
\begin{proof}
The set $S\coloneqq\{ (a,b)\in\er^{d-1}\times\er: b\leq g(a) \}$
has positive reach by \cite[Theorem~2.3]{Fu} and consequently $d_{\graph g}=d_{\partial S}$ is DC by Proposition~\ref{doplpr}.
\end{proof}

\begin{rem}\label{hrsemi}
	Let $M\subset\er^d$ be a closed set whose boundary can be locally expressed as a graph of a semiconvex function (i.e., for each $a \in \partial M$ there exist 
	 a semiconcave function $g:\er^{d-1}\to\er$, $\delta>0$ and an isometry $\vf: \R^d \to \R^d$
	 such that $\partial M \cap U(a,\delta) = \vf(\graph g) \cap U(a,\delta)$).
	Then $d_{\partial M}$ is locally DC (and therefore DC) by Corollary~\ref{grsemi} and \eqref{loksem} and so $M\in\cal D_d$ by Lemma~\ref{vldc} (iv) and \eqref{eq:boundaryEquivalenceGeneral}.
\end{rem}

%Next, we present several observations concerning stability of classes $\cal D_d$.
%
%
%
%
%Using \eqref{cald1}, it is easy to see that, for  $d=1$,
%\begin{equation}\label{prunik}
%\cal D_d\ \ \text{is closed with respect to finite intersections}.
%\end{equation}
%
%Note also that, for $d=1$, the implication
%\begin{equation}\label{eq:complement}
%A \in \cal D_d \Rightarrow \overline{\R^d \setminus A} \in \cal D_d
%\end{equation}
%holds by \eqref{cald1}.
%
%However, neither \eqref{prunik} nor \eqref{eq:complement} does hold if $d\geq 2$. 
%We will present a counterexample
% in $\R^2$ only, but a similar construction works also in the general case.
% 

%A set $M\subset\er^d$ is called a DC graph if there is a direction $v\in S^1$, a closed convex set $K\subset v^{\bot}$
%and a DCR function $f:K\to \lin (v)$ such that
%\begin{equation*}
%M=\{t+f(t):t\in K\}.
%\end{equation*}

%Let $C\subset \R$ be the Cantor ternary set and $f(x)\coloneqq  (d_C(x))^2,\ x \in \R$. Then 
% $f$ is a DC function by \eqref{odist}. Set $A\coloneqq  \graph f$ and $B\coloneqq  \graph (-f)$. Then
% $A$, $B$ belong to $\cal D_2$ by  Theorem \ref{hlav}. On the other hand, since
% $A\cap B = C \times \{0\}$, we easily obtain $A\cap B \notin \cal D_2$ using \eqref{cald1}. 

Before the next results, we present the following definitions:
we say that a set $A\subset \er^d$ is a DC hypersurface, if there exist a vector $v\in S^{d-1}$ and a DC function (i.e. the difference of two convex functions) $g$ on $W\coloneqq (\spn v)^\bot$ such that $A=\{w+g(w)v: w\in  W\}$.
A set $P\subset\er^2$ will be called a DC graph if it is a rotated copy of $\graph(f|_{I})$ for a DC function $f:\er\to\er$ and some compact (possibly degenerated) interval $\varnothing\not=I\subset\er$.
Note that $P$ is a DC graph if and only if it is a nonempty connected compact subset of a DC hypersurface in $\er^2$.

\begin{prop}\label{pokr}
	Let $d\geq 2$ and $F \in \cal D_d$. Then each bounded set $C \subset \partial F$ can be covered by finitely many DC hypersurfaces.
\end{prop}
\begin{proof}
	By our assumptions, $f\coloneqq  \dist(\cdot, F)$ is a DC function on $\R^d$ and $f(x)=0$
	for every $x \in C$. So, by \cite[Crollary 5.4]{PRZ} it is sufficient to prove that
	for each $x \in C$ there exists $y^* \in \partial f(x)$ with  $|y^*|> \ep\coloneqq  1/4$, where
	$\partial f(x)$ is the Clarke generalized gradient of $f$ at $x$ (see \cite[p. 27]{C}).
	To this end, suppose to the contrary that $x \in C$ and $\partial f(x)\subset B(0,1/4)$.
	Since the mapping $x \mapsto \partial f(x)$ is upper semicontinuous (see \cite[Proposition 2.1.5 (d)]{C}), there exists $\delta>0$ such that $\partial f(u) \subset U(0, 1/2)$ for each $u \in U(x, \delta)$.
	Since $x \in \partial F$, we can choose $z \in U(x,\delta/2) \setminus F$ and $p \in \Pi_F(z)$.
	Then $p \in U(x, \delta)$, $f(z)- f(p) = |z-p|$ and Lebourg's mean-value theorem
	(see \cite[Theorem 2.3.7]{C}) implies that there exist $u \in U(x,\delta)$ and $\alpha \in
	\partial f(u)$ such that
	$$ \langle \alpha, z-p\rangle = f(z)- f(p) = |z-p|.$$ 
	Therefore  $|\alpha| \geq 1$, which is a contradiction.
\end{proof}

%A set $M\subset\er^2$ is called a DC graph if there is a direction $v\in S^1$, a closed interval, set $K\subset v^{\bot}$
%and a DC function $f:v^{\bot}\to \lin (v)$ such that
%\begin{equation*}
%M=\{t+f(t):t\in K\}.
%\end{equation*}

%We will say that a set $A$ is locally a finite union of DC graphs if for every $x\in A$ there is $r>0$ and a set $M$ which is a finite union of DC graphs such that $A\cap U(x,r)=M\cap U(x,r)$.

The above proposition easily implies the following fact.
\begin{cor}
	If $F\in \cal D_2$ then $\partial F$ is a subset of the union of a locally finite system of DC graphs.
\end{cor}

Using Theorem~\ref{hlav} we obtain the following easy result.
\begin{prop}\label{prop:unionOfGraphs}
	If $A\subset\er^2$ is the union of a locally finite system of DC graphs then $A\in \cal D_2$.
%	Let $P_1,\dots,P_N$ be DC graphs,
%	then $P_1\cup \cdots\cup P_N\in \cal D_2$.
\end{prop}

\begin{proof}
First note that it is enough to prove that any DC graph $P$ belongs to $\cal D_2$. 
Indeed, if $M$ is a locally finite system of DC graphs and each DC graph belongs to $\cal D_2$, then $d_M$ is locally DC  by \eqref{sjedno} (and so DC) and $M\in\cal D_2$.

So assume that $A$ is a DC graph.
Without any loss of generality we may assume that $A=\graph f|_{[0,p]}$ for some DC function $f:\er\to\er$. If $p=0$ then $d_A=|\cdot|$ is even convex, so assume that $p>0$. We may also assume that $f(0)=0$.

First note that (by Theorem~\ref{hlav} and \eqref{loksem}) $d_A$ is locally DC on $\er^2\setminus\{(0,0),(p,f(p))\}$.
It remains to prove that $d_A$ is DC on some neighbourhood of $(0,0)$ and $(p,f(p))$. We will prove only the case of the point $(0,0)$, the other case can be proved quite analogously.
By Lemma \ref{vldc} (iii) we can choose $L>0$ such that $f$ is $L$-Lipschitz
on $[0,p]$.
Define
\begin{equation*}
f_\pm(x)\coloneqq
\begin{cases}
f(x)& \text{if}\quad 0\leq x \leq p,\\
f(p)& \text{if}\quad p<x,\\
\pm 2Ly& \text{if}\quad x<0.
\end{cases}
\end{equation*}

It is easy to see that both $f_+$ and $f_-$ are continuous and so they
are DC by Lemma \ref{vldc} (iv).

Put 
\begin{equation*}
M_0\coloneqq \left\{ (u,v)\in\er^2: u\geq0, \; v=f_+(u)\right\},
\end{equation*}
\begin{equation*}
M_1\coloneqq \left\{ (u,v)\in\er^2: u\geq0, \; f_+(u)< v\right\}\cup
\left\{ (u,v)\in\er^2: u<0, \; -\frac{u}{2L}< v\right\},
\end{equation*}
\begin{equation*}
M_2\coloneqq \left\{ (u,v)\in\er^2: u\geq0, \; \tilde f_-(u)> v\right\}\cup
\left\{ (u,v)\in\er^2: u<0, \; \frac{u}{2L}> v\right\}
\end{equation*}
and
\begin{equation*}
M_3\coloneqq 
\left\{ (u,v)\in\er^2:\frac uL< v< -\frac uL\right\}.
\end{equation*}
Clearly $\er^2= M_0\cup M_1\cup M_2\cup M_3$ and  $M_1$, $M_2$, $M_3$ are open.

Set  $\tilde d\coloneqq\dist(\cdot, M_0)$ and, for each $y \in \R^2$, define
$$d_0(y)=0,\ \  d_1(y)\coloneqq \dist(y,\graph f_+),\ \
d_2(y)\coloneqq \dist(y,\graph f_-),
\ \ d_3(y)\coloneqq |y|.$$
Functions $d_1$ and $d_2$ are DC on $\er^2$ by Theorem~\ref{hlav}, $d_0$ and $d_3$ are even convex on $\er^2$.

Using (for $K= 1/L, -1/L, 1/(2L), -1/(2L)$) the facts that the lines with the slopes $K$ and $-1/K$
are orthogonal and $M_0\subset \{(u,v): u\geq 0,\; -Lu\leq v \leq Lu \}$, 
easy geometrical observations show that
\begin{equation}\label{mjdt}
\tilde d(y)= d_i(y)\ \ \text{ if}\ \  y \in M_i,\  0\leq i \leq 3,
\end{equation}
and so Lemma \ref{vldc} (iv)  implies that $\tilde d$ is DC.
To finish the proof it is enough to observe that $d_A=\tilde d$ on $U(0,\frac{p}{2})$.
\end{proof}

%\begin{rem}
%	Note that many examples on sets from $\cal D_2$ with empty interior are of the form $P_1\cup \cdots\cup P_N$ for some DC graphs $P_1,\dots,P_N$.
%	However, the next example shows that not all sets in $\cal D_2$ with empty interior are of this form.
%\end{rem}

However, the following example shows that the opposite implication does not hold even for nowhere dense sets $A$.

\begin{ex}
	There is a  nowhere dense set $A\in {\cal D}_2$ which is not the union of a locally finite system of DC graphs.
\end{ex}
\begin{proof}
	Define 
	$$\text{$f(x)= \max(x^5,0)$, $x \in \R$, $g(x)=x^5\cos\frac{\pi}{x}$, $x \in \R$, and 
	$g_k\coloneqq  g\restriction_{[\frac{1}{2k+1},\frac{1}{2k}]}$, $k \in \en $.}$$ 
	Put $A^{\pm}\coloneqq \graph (\pm f)$, $A_k\coloneqq \graph g_k$, $k\in\en$, and
	\begin{equation*}
	A\coloneqq A^+\cup A^-\cup \bigcup_{k\in \en} A_k.
	\end{equation*}
	$A$ is clearly closed and nowhere dense, and it 
	   is not the union of a locally finite system of DC graphs since every DC graph $B\subset A$ can intersect at most one of the sets $A_i$.
	It remains to prove that $A\in {\cal D}_2$.	
		
First we will describe all components of $\R^2 \setminus A$. To this end, for each $k \in \en$, define

\begin{equation*}
	U_0(x)=
	\begin{cases}
	g(x), &\; x\in [1/3,1/2],\\
	f(x), &\; x\in [1/2,\infty)
	,
	\end{cases},
	\quad
	U_k(x)=
	\begin{cases}
	g(x), &\, x\in \left[\frac{1}{2k+3},\frac{1}{2k+2}\right],\\
	f(x), &\, x\in \left[\frac{1}{2k+2},\frac{1}{2k}\right],
	\end{cases}
	\end{equation*}
	
	\begin{equation*}
	L_0(x)=
	-f(x),  x\in [1/3,\infty),
	\quad
	L_k(x)=
	\begin{cases}
	-f(x), &\, x\in \left[\frac{1}{2k+3},\frac{1}{2k+1}\right],\\
	g(x), &\, x\in \left[\frac{1}{2k+1},\frac{1}{2k}\right].
	\end{cases}
	\end{equation*}
	Set $G_k\coloneqq  \{(x,y):\ L_k(x) < y <U_k(x)\}$, $ k=0,1,2,\dots$. Then
	$$ G^+\coloneqq   \{(x,y):\ f(x) < y \}, G^-\coloneqq   \{(x,y):\ y< -f(x)  \}\ \text{and}\  G_0,G_1,\dots$$
	 are clearly all components of $\R^2 \setminus A$.
	
	Recall that both $U_k$ and $L_k$ is defined on $D_k$, where $D_k= [\frac{1}{2k+3}, \frac{1}{2k}]$
	 for $k\in \en$ and $D_0= [1/3, \infty)$. Using the facts that $D_k$ and $D_{k+2}$ are disjoint
	 ($k=0,1,\dots$),
	$$ U_k\left(\frac{1}{2k+3}\right)= L_k\left(\frac{1}{2k+3}\right) = g\left(\frac{1}{2k+3}\right),\ 
	U_k\left(\frac{1}{2k}\right)= L_k\left(\frac{1}{2k}\right) = g\left(\frac{1}{2k}\right)$$
	and $U_0(1/3)= L(1/3)=g(1/3)$, it is easy to see that there exist unique functions
	$U$, $\tilde U$ which are continuous on $\R$, $U$ (resp. $\tilde U$) extends all
	 $U_k$, $k=0,2,4,\dots$ (resp. $k=1,3,5,\dots$) and $U$ (resp. $\tilde U$) equals to $g$
	 at all points at which no $U_k$, $k=0,2,4,\dots$ (resp. $k=1,3,5,\dots$) is defined.
	Quite analogously  a continuous function $L$ (resp. $\tilde L$) extending all
	$L_k$, $k=0,2,4,\dots$ (resp. $k=1,3,5,\dots$) is defined. Since the functions
	 $g$, $f$, $-f$ are DC,
	 Lemma \ref{vldc} (iv) implies that 
	the functions $U$, $\tilde U$, $L$, $\tilde L$ are DC. So Theorem \ref{hlav} implies that
	 the distance functions
	\begin{equation}\label{vzdal}
	d_{A^+},\  d_{A^-},\ d_{\graph U},\ d_{\graph \tilde U},\  d_{\graph L},\ d_{\graph \tilde L}
	\end{equation}
	 are DC.
	
	Obviously $d_A(x)=0$ for $x\in A$, $d_A(x)=d_{A^+}(x)$ for $x\in G^+$ and
	 $d_A(x)=d_{A^-}(x)$ for $x\in G^-$. Further, if $x \in G_k$ with $k=2,4,6,\dots$, then
	$$d_A(x) \in \{  d_{\graph U}(x), d_{\graph L}(x)\},$$ 
	 which easily follows from the facts that 
	$$ \partial G_k \subset  (\graph U \cup \graph L)\ \ \text{ and}\ \
	 (\graph U \cup \graph L) \cap G_k = \varnothing.$$
	 Similarly we obtain that,  if $x \in G_k$ with $k=1,3,5,\dots$, then
	$$d_A(x) \in \{ d_{\graph \tilde U}(x), d_{\graph \tilde L}(x)\}.$$
	 Thus, using \eqref{vzdal} and  Lemma~\ref{vldc} (iv), we obtain that
	 $d_A$ is DC.
\end{proof}

It seems that there does not exist an essentially simpler example. Iterating the construction of the example we can obtain nowhere dense sets in $\cal D_2$ of quite complicated topological structure.

In our opinion, using Proposition~\ref{pokr} and Theorem~\ref{hlav} it is possible to give an optimal complete characterisation of sets in $\cal D_2$, but it appears to be a rather hard task.
We believe that we succeeded to find some characterisation, however, it is not quite satisfactory and our current proof is very technical.
We aim to find a better characterisation, hopefully with a simpler proof.

%We believe we have a full characterisation of sets in $\cal D_2$, unfortunately, it is quite complicated and not very satisfactory and the proof is very technical, We aim to find more satisfactory characterisation, hopefully with a simpler proof, in the future.

\end{document}